\documentclass[11pt]{amsart}
\usepackage{amsmath,amssymb,amsthm,mathtools, enumitem}
\usepackage{hyperref}
\hypersetup{hidelinks}
\usepackage{cite}

\usepackage{centernot}

\DeclareMathOperator{\re}{Re}

\newtheorem{theorem}{Theorem}[section]
\newtheorem{lemma}[theorem]{Lemma}
\newtheorem{proposition}[theorem]{Proposition}
\newtheorem{corollary}[theorem]{Corollary}

\theoremstyle{remark}
\newtheorem{remark}[theorem]{Remark}

\theoremstyle{definition}
\newtheorem{definition}{Definition}

\title{Weak stability and binary convex combinations of slices}

\author{Rainis Haller}
\address{Institute of Mathematics and Statistics, University of Tartu, Estonia}
\email{rainis.haller@ut.ee}


\thanks{This work was supported by the Estonian Research Council grant PRG1901.}
\date{}

\begin{document}

\begin{abstract}
    We study weak openness and related geometric properties of convex combinations of slices of the unit balls of Banach spaces. We prove that, for every $m\geq 2$, properties $\mathrm{P1}^{(m)}$ and $\mathrm{CWO}^{(m)}$ are equivalent, and likewise for their closure variants. Consequently, both $\mathrm{P1}$ and $\overline{\mathrm{P1}}$ are determined by convex combinations of two slices; in particular P1 and CWO are equivalent. We construct infinite-dimensional real Banach spaces with $\mathrm{P2}^{(2)}$ but without P2, and with $\mathrm{P3}^{(2)}$ but without P3. In the latter space, every convex combination of two slices contains two points at distance 2, although the space fails the strong diameter two property. Finally, we give an equivalent norm on the real space $c_0$ for which  $\overline{\mathrm{P1}}$ holds, but P1 fails. These answers several questions raised in the literature.
\end{abstract}


\maketitle
\section{Introduction}
Let $X$ be a real or complex Banach space. We denote its closed unit ball and unit sphere by $B_X$ and $S_X$, respectively. A slice of $B_X$ is a set of the form $S(B_X,x^*,\alpha)=\{x\in B_X : \re x^*(x)>1-\alpha\}$ with $x^*\in S_{X^*}$ and $\alpha>0$.

\begin{definition}
A Banach space $X$ has
\begin{enumerate}[label=\textup{(\roman*)}]
    \item property P1 if every finite convex combination of slices of $B_X$ is relatively weakly open in $B_X$;
    \item property CWO if every finite convex combination of non-empty relatively weakly open subsets of $B_X$ is relatively weakly open;
    \item property $\overline{\text{P1}}$ if, whenever $C$ is a finite convex combination of slices of $B_X$, and $x\in C$, there exists a relatively weakly open $W\subset B_X$ such that $x\in W\subset \overline C^{\|\cdot\|}$;
    \item property $\overline{\textup{CWO}}$ if the preceding condition holds for every finite convex combination $C$ of non-empty relatively weakly open subsets of $B_X$;
    \item property P2 if every finite convex combination of slices has non-empty relative weak interior in $B_X$;
    \item property P3 if every finite convex combination of slices intersects $S_X$.
\end{enumerate}
For $m\geq 2$, the superscript $(m)$ means that only convex combinations of at most $m$ sets are considered.
\end{definition}

By Bourgain's lemma, every non-empty relatively weakly open subset of $B_X$ contains a convex combination of slices of $B_X$ \cite[Lemma~II.1]{MR912637}. Since slices are relatively weakly open, CWO implies P1, and clearly, P1 implies $\overline{\mathrm{P1}}$. If $X$ is infinite-dimensional, every non-empty relatively weakly open subset of $B_X$ intersects $S_X$. Hence, in an infinite-dimensional space,
\[\textup{P1}\implies \textup{P2}\implies \textup{P3}\qquad (\operatorname{dim}X=\infty).\]

The properties P1, P2, and P3 were formulated in \cite[Section~3]{MR3834668}, and the notation was introduced in \cite{MR4102873}. The same conditions are denoted in \cite{MR3994868} by (W1), (W2), and (CS), respectively. Property CWO was introduced in \cite{MR4034749}; it is also referred to as weak stability of the unit ball in \cite{MR4422399}. The norm-closure variant $\overline{\text{P1}}$ was introduced in \cite{MR4116184}.

Our first result shows that:
\[
\textup{P1}^{(m)}\iff \textup{CWO}^{(m)}\qquad\text{and}\qquad \overline{\textup{P1}}^{(m)}\iff\overline{\textup{CWO}}^{(m)}.
\]
It follows that 
\[
\textup{P1}^{(2)}\implies \textup{CWO}^{(2)}\implies \textup{CWO}\implies \textup{P1}.
\]
The equivalence $\textup{P1}=\textup{CWO}$ answers a question of L\'opez-P\'erez and Medina \cite[after Corollary~2.4]{MR4422399}.

In contrast, we construct real Banach spaces, showing that the analogous binary assertions fail at the next level:
\[
\textup{P2}^{(2)}\centernot\implies \textup{P2}\qquad\text{and}\qquad \textup{P3}^{(2)}\centernot\implies \textup{P3}.
\]
In the latter example, every convex combination of two slices contains two points at distance $2$, although the space fails the strong diameter two property. Recall that a space has the strong diameter two property (SD2P) if every finite convex combination of slices of its unit ball has diameter $2$. In the real case, P3 is equivalent to requiring that every such combination contain two points at distance $2$ \cite[Theorem~3.4]{MR3994868}.

Finally, we construct an equivalent norm on the real space $c_0$ for which $\overline{\mathrm{P1}}$ holds but P1 fails. The mean of two opposite slices already witnesses the failure. This separates, for general Banach spaces, two conditions whose relationship was left open in \cite{MR4116184} (see also \cite[p.~7]{MR4422399}).

\section{P1, CWO, and their binary versions}

We start with an elementary lemma.

\begin{lemma}\label{lem:2.1}
    Let $A\subset \mathbb R^N$ be non-empty, convex, and bounded, and put 
    \[\Delta_N=\{c=(c_1,\dotsc,c_N)\in [0,1]^N : \sum_{i=1}^N c_i=1\}.\] 
    If $\sup_{a\in A} c\cdot a>0$ for every $c\in \Delta_N$, then $A\cap (0,\infty)^N\neq \varnothing$.
\end{lemma}
Here $c\cdot a$ denotes the usual scalar product on $\mathbb R^N$.

\begin{proof}
    Suppose that $A$ and $(0,\infty)^N$ were disjoint. The separation theorem gives a non-zero $d\in\mathbb R^N$ such that
    \[
    \sup_{a\in A}d\cdot a\leq \inf_{b\in (0,\infty)^N} d\cdot b.
    \]
    The infimum on the right is finite only if every coordinate of $d$ is non-negative, and then this infimum is equal to $0$. Normalising $d$ by the sum of its coordinates contradicts the assumption.
\end{proof}

\begin{theorem}\label{thm:thm2.1}
    For every Banach space $X$ and every integer $m\geq 2$, 
    \[\textup{P1}^{(m)}\iff \textup{CWO}^{(m)}\qquad\text{and}\qquad \overline{\textup{P1}}^{(m)}\iff \overline{\textup{CWO}}^{(m)}.\]
\end{theorem}

\begin{proof}
    The implications $\textup{CWO}^{(m)}\implies \textup{P1}^{(m)}$ and $\overline{\textup{CWO}}^{(m)}\implies \overline{\textup{P1}}^{(m)}$ are immediate. 
    
    We prove the converse implications $\textup{P1}^{(m)}\implies \textup{CWO}^{(m)}$ and $\overline{\textup{P1}}^{(m)}\implies \overline{\textup{CWO}}^{(m)}$ simultaneously; we refer to them as the exact and closure cases, respectively. The assertion is trivial for $X=\{0\}$, so assume that $X\neq\{0\}$. 
    
    Fix $n\leq m$, non-empty relatively weakly open sets $U_i\subset B_X$, $\lambda_i>0$ with $\sum_{i=1}^n\lambda_i=1$, and a point $u=\sum_{i=1}^n\lambda_i u_i$ with $u_i\in U_i$.
    
    For every $i$, choose a basic neighbourhood 
    \[u_i\in U_i^0=\{x\in B_X : \re f_{ij}(x)>\alpha_{ij},\ j=1,\dotsc,m_i\}\subset U_i.\]
    Let 
    \[
    \gamma_{ij}=\re f_{ij}(u_i)-\alpha_{ij}>0\qquad\text{and}\qquad r_{ij}(x)=\frac{\re f_{ij}(x)-\alpha_{ij}}{\gamma_{ij}}.
    \]
    Thus $r_{ij}(u_i)=1$, and an element $x\in B_X$ belongs to $U_i^0$ precisely when all the numbers $r_{ij}(x)$ are positive.

    Let $N=\sum_{i=1}^n m_i$. For $c=(c_{ij})\in \Delta_N$, define
    \[
    g_i^c=\sum_{j=1}^{m_i}\frac{c_{ij}}{\gamma_{ij}}f_{ij}
    \qquad\text{and}\qquad
    S_i^c=\{x\in B_X : \re g_i^c(x)>\re g_i^c(u_i)-\tfrac{1}{2n}\}.
    \]
    Each $S_i^c$ is a slice containing $u_i$. Consider 
    \[V_c=\sum_{i=1}^n\lambda_i S_i^c.\]
    In the exact case, $V_c$ is relatively open by $\textup{P1}^{(m)}$. In the closure case, $\overline{\textup{P1}}^{(m)}$ gives a relatively weakly open set $W_c$ such that $u\in W_c\subset \overline{V_c}^{\|\cdot\|\}}$.

    Choose $M<\infty$ such that, for every $a\in B_X$, $i=1,\dotsc,n$, and $j=1,\dotsc,m_i$, $|r_{ij}(a)|\leq M$. Let $0<\delta<1/(2M)$ and let $D$ be a finite $\delta$-net of $\Delta_N$ in the $\ell_1$-norm. Then $W=\bigcap_{d\in D}V_d$, in the exact case, and $W=\bigcap_{d\in D}W_d$, in the closure case, is a relatively weakly open neighbourhood of $u$ in $B_X$. Fix $x\in W$. In the exact case, let
    \[A=\big\{(r_{ij}(x_i))_{i,j} : x_i\in B_X,\ \sum_{i=1}^n \lambda_i x_i=x\big\}.\]
    In the closure case, for $\varepsilon>0$, let
    \[
    A_{\varepsilon}=\big\{(r_{ij}(x_i))_{i,j} : x_i\in B_X,\ \Big\|\sum_{i=1}^n \lambda_i x_i-x\Big\|<\varepsilon\big\}.
    \]
    Let $c\in \Delta_N$, and choose $d\in D$ with $\|c-d\|_1<\delta$. In the exact case, $x\in V_d$, so there are $x_i\in S_i^d$ ($i=1,\dotsc,n$) such that $\sum_{i=1}^n\lambda_ix_i=x$. In the closure case, $x\in W_d$, and, for every $\varepsilon>0$, there are $x_i\in S_i^d$ ($i=1,\dotsc,n$) such that $\|\sum_{i=1}^n\lambda_i x_i-x\|<\varepsilon$. For the corresponding vector $a=(r_{ij}(x_i))_{i,j}$ we have 
    \[
    d\cdot a=1+\sum_{i=1}^n\re \big(g_i^d(x_i)-g_i^d(u_i)\big)>1-n\tfrac{1}{2n}=\tfrac12.
    \]
    Consequently, 
    \[
    c\cdot a\geq d\cdot a-\|c-d\|_1\|a\|_\infty>\tfrac12-\delta M>0.
    \]
    Lemma~\ref{lem:2.1} gives a vector whose coordinates are positive.

    In the exact case, this gives a representation $x=\sum_{i=1}^n\lambda_i x_i$ with $x_i\in U_i^0$ for every $i$. Hence $x\in \sum_{i=1}^n\lambda_i U_i$, and therefore
    \[W\subset \sum_{i=1}^n\lambda_i U_i.\]
    In the closure case, we get
    \[W\subset \overline{\sum_{i=1}^n\lambda_i U_i}^{\|\cdot\|}.\]
    This proves both converse implications.
\end{proof}

\begin{corollary}\label{cor:cor2.3}
    For every Banach space $X$ and every integer $m\geq 2$,
    \[
    \textup{P1}^{(m)}\iff \textup{P1}\qquad\text{and}\qquad \overline{\textup{P1}}^{(m)}\iff \overline{\textup{P1}}.
    \]
\end{corollary}
\begin{proof}
    By Theorem~\ref{thm:thm2.1}, $\textup{P1}^{(2)}$ implies $\textup{CWO}^{(2)}$. Convex combinations of two non-empty relatively weakly open subsets of $B_X$ are therefore relatively weakly open. Arguing by induction shows that $X$ has CWO (see also \cite[Lemma~4.1 (a)]{MR4034749}), and hence P1.
    For the closure version, the proof is essentially the same, and we omit it here.
\end{proof}

\begin{remark}
    In \cite[after Corollary~2.4]{MR4422399}, L\'opez-P\'erez and Medina asked whether weak stability of $B_X$ and the relative weak openness of every convex combination of slices of $B_X$ are equivalent for arbitrary Banach spaces. Theorem~\ref{thm:thm2.1} answers this question affirmatively.
\end{remark}

\begin{remark}
    The proof of Theorem~\ref{thm:thm2.1} uses only basic neighbourhoods determined by finitely many functionals. It therefore applies verbatim to a dual space endowed with the weak-star topology, once slices and relatively weakly open sets are replaced by weak-star slices and relatively weak-star open sets. Consequently, the weak-star analogues of Theorem~\ref{thm:thm2.1} and Corollary~\ref{cor:cor2.3} hold as well.
\end{remark}

\section{The binary versions of P3 and P2}

We first fix some notation. Let $2\leq n\in\mathbb N$. A norm $N$ on $\mathbb R^n$ is called \emph{absolute normalised} if 
\[N(a_1,\dotsc,a_n)=N(|a_1|,\dotsc,|a_n|)\qquad\text{and}\qquad N(e_i)=1,\]
for all $(a_1,\dotsc,a_n)\in\mathbb R^n$ and every standard unit vector $e_i$. For Banach spaces $X_1,\dotsc,X_n$, we write ${(\bigoplus_{i=1}^n X_i)}_N$ for their direct sum equipped with the norm $\|(x_1,\dotsc,x_n)\|_N=N(\|x_1\|,\dotsc,\|x_n\|)$. We also set $B_N^+=B_{(\mathbb R^n,N)}\cap [0,\infty)^n$.


\begin{proposition}\label{prop:P3^2,aga_mitte_P3}
    There is an infinite-dimensional real Banach space with $\textup{P3}^{(2)}$ but without $\textup{P3}$. 
\end{proposition}
\begin{proof}
    Define an absolute normalised norm on $\mathbb R^3$ by 
    \[
    N(r,s,t)=\max\{|r|+|s|,|r|+|t|,|s|+|t|\}.
    \]
    Geometrically, the unit ball of $(\mathbb R^3, N)$ is a rhombic dodecahedron. In the first octant, its vertices are $e_1,e_2,e_3$, and $(\frac12,\frac12,\frac12)$. Every two of these vertices lie on a common facet, and thus the line segment joining them is contained in $S_{(\mathbb R^3,N)}$.
    
    Let $Z={(\bigoplus_{i=1}^3 c_0)}_N$. Consider two slices \[S\coloneqq S(B_Z,(f_1,f_2,f_3),\alpha)\qquad\text{and}\qquad T\coloneqq S(B_Z,(g_1,g_2,g_3),\beta),\]
    and $\lambda\in(0,1)$. 
    Choose vertices $a,b\in B_N^+$ such that 
    \[\sum_{i=1}^3a_i\|f_i\|=1\qquad \text{and}\qquad \sum_{i=1}^3b_i\|g_i\|=1.\]
    There are slices $S_i$ and $T_i$ of $B_{c_0}$ such that 
    \[
    (a_1 x_1,a_2x_2, a_3 x_3)\in S\qquad\text{and}\qquad(b_1y_1,b_2y_2,b_3y_3)\in T
    \]
    whenever $x_i\in S_i$ and $y_i\in T_i$ for $i=1,2,3$. If the coefficient or the corresponding functional is zero, then the corresponding slice can be chosen arbitrarily.

     Let $c=\lambda a+(1-\lambda) b$. Then $N(c)=1$. Write $c=(c_1,c_2,c_3)$. By \cite[Theorem~2.4]{MR3834668}, $c_0$ has P1 and hence P3. Thus, whenever $c_i>0$, we can choose $x_i\in S_i$ and $y_i\in T_i$ such that \[w_i\coloneqq\frac{\lambda a_i}{c_i}x_i+\frac{(1-\lambda)b_i}{c_i}y_i\in S_{c_0}.\] If $c_i=0$, then choose any $w_i\in S_{c_0}$. Then
    \[
    z\coloneqq (c_1w_1,c_2w_2,c_3w_3)\in \lambda S+(1-\lambda)T
    \]
    and $\|z\|=N(c)=1$. Thus $Z$ has $\textup{P3}^{(2)}$.
    
    To see that $Z$ fails P3, fix $f_1,f_2,f_3\in S_{c_0^*}$ and $\alpha\in(0,1/4)$. Consider three slices \[S_1\coloneqq S(B_Z,(f_1,0,0),\alpha),\qquad S_2\coloneqq S(B_Z,(0,f_2,0),\alpha),\]\[\text{and}\qquad S_3\coloneqq S(B_Z,(0,0,f_3),\alpha).\] 
    If $x=(x_1,x_2,x_3)\in S_i$, then $\|x_j\|<\alpha$ for $j\neq i$. Hence every $z=(z_1,z_2,z_3)\in (S_1+S_2+S_3)/3$ satisfies $\|z_j\|\leq (1+2\alpha)/3$ ($j=1,2,3$), and therefore $\|z\|\leq 2(1+2\alpha)/3<1$. Thus $Z$ fails P3.    
\end{proof}

\begin{remark}
    The space $Z$ in Proposition~\ref{prop:P3^2,aga_mitte_P3} does not have the SD2P because the three-slice convex combination used above is contained in $rB_Z$ for some $r<1$. Nevertheless, every convex combination of two slices of $B_Z$ contains two points at distance $2$. Indeed, we use the notation from the first part of the proof. If $c_i>0$, then \cite[Theorem~3.4]{MR3994868} gives two points \[w_i^1,w_i^2\in  \frac{\lambda a_i}{c_i}S_i+\frac{(1-\lambda) b_i}{c_i} T_i\] such that $\|w_i^1-w_i^2\|=2$. If $c_i=0$, then let $w_i^1=w_i^2\in B_{c_0}$. We then obtain, for $j=1,2$, 
     \[z^{j}\coloneqq (c_1w_1^{j},c_2w_2^{j},c_3w_3^{j})\in \lambda S+(1-\lambda)T,\]
    and \[\|z^1-z^2\|=N(2c_1,2c_2,2c_3)=2.\] 
    To the best of our knowledge, it was previously unknown whether convex combinations of two slices determine SD2P; our example shows that this is not the case.
\end{remark}

\begin{remark}
    The same space in Proposition~\ref{prop:P3^2,aga_mitte_P3} does not have P2$^{(2)}$. Let $e_1^*$ be the first coordinate functional on $c_0$ and put 
    \[S\coloneqq S(B_Z,(e_1^*,0,0),1/4)\qquad\text{and}\qquad T\coloneqq S(B_Z,(0,e_1^*,0),1/4).\] 
    If $z=(z_1,z_2,z_3)\in (S+T)/2$, then 
    \[\|z_1\|<5/8,\qquad \|z_2\|<5/8,\qquad \text{and}\qquad \|z_3\|<1/4.\]
    For all sufficiently large $n$, the point $z^{(n)}=(z_1,z_2,z_3+e_n/3)$ belongs to $B_Z$ because $\|z_3+e_n/3\|<3/8$ and $\|z_1\|,\|z_2\|<5/8$, while $\|z_1\|+\|z_2\|\leq 1$. Since $z^{(n)}\to z$ weakly and $\|z_3+e_n/3\|>1/4$, we have $z^{(n)}\notin(S+T)/2$, so $(S+T)/2$ has empty relative weak interior.
\end{remark}

A slight modification of the preceding example separates P2$^{(2)}$ and P2.

\begin{proposition}\label{prop:P2}
    There is an infinite-dimensional real Banach space with $\textup{P2}^{(2)}$ but without $\textup{P3}$.
\end{proposition}

\begin{proof}
    Define an absolute normalised norm on $\mathbb R^3$ by
    \[
    N(r,s,t)=\max\{|r|+|s|+\frac12|t|,|r|+\frac12|s|+|t|,\frac12|r|+|s|+|t|\},
    \]
    and let $Z={(\bigoplus_{i=1}^3 c_0)}_N$. The non-zero vertices of $B_N^+$ are $e_1,e_2,e_3$ and $v=(\frac25,\frac25,\frac25)$.
    Let \[F_{12}\coloneqq\operatorname{conv}\{e_1,e_2,v\},\qquad F_{13}\coloneqq\operatorname{conv}\{e_1,e_3,v\},\]\[\text{and}\qquad F_{23}\coloneqq \operatorname{conv}\{e_2,e_3,v\}.\]

    We first prove that $Z$ has P2$^{(2)}$. Consider two slices \[S\coloneqq S(B_Z,(f_1,f_2,f_3),\alpha)\qquad\text{and}\qquad T\coloneqq S(B_Z,(g_1,g_2,g_3),\beta),\] and $\lambda\in(0,1)$. Choose vertices $a,b\in B_N^+$ such that \[\sum_{i=1}^3 a_i\|f_i\|=1\qquad\text{and}\qquad \sum_{i=1}^3 b_i\|g_i\|=1.\] Some facet among $F_{12},F_{13},F_{23}$ contains both $a$ and $b$. By symmetry, assume that it is $F_{12}$. Choose $\widetilde a$, $\widetilde b$ in the relative interior of $F_{12}$, close to $a$ and $b$, so that 
    \[\sum_{i=1}^3 \widetilde a_i\|f_i\|>1-\alpha\qquad \text{and}\qquad \sum_{i=1}^3 \widetilde b_i\|g_i\|>1-\beta.\]
    Choose $\eta>0$ and slices $S_i$ and $T_i$ of $B_{c_0}$ such that
    \[
    \sum_{i=1}^3\widetilde a_i f_i(u_i)>1-\alpha+\eta\qquad\text{and}\qquad
    \sum_{i=1}^3\widetilde b_i g_i(v_i)>1-\beta +\eta,
    \]
    whenever $u_i\in S_i$ and $v_i\in T_i$. 
    
    Let $c=\lambda \widetilde a+(1-\lambda)\widetilde b$. Then $c\in F_{12}$, so $N(c)=1$, and all coordinates of $c$ are positive. For $i=1,2,3$, let 
    \[W_i=\frac{\lambda \widetilde a_i}{c_i}S_i+\frac{(1-\lambda)\widetilde b_i}{c_i}T_i.\]
    By \cite[Theorem~2.4]{MR3834668}, $c_0$ has P1; hence each $W_i$ is a non-empty relatively weakly open subset of $B_{c_0}$ and therefore intersects the unit sphere. Let $w_i\in W_i\cap S_{c_0}$ for $i=1,2,3$. Then 
    \[z^0\coloneqq (c_1w_1,c_2w_2,c_3w_3)\in \lambda S+(1-\lambda)T\]
    and $\|z^0\|=N(c)=1$. 
    We next show that $z^0$ belongs to the relative weak interior of $\lambda S+(1-\lambda)T$.

    Let $d=(1,1,\frac12)$. Then $F_{12}=\{e\in B_N^+:d\cdot e=1\}$. Pick $\kappa>0$ such that $\tilde a+p\in B_N^+$ and $\tilde b+p\in B_N^+$ whenever $\|p\|_\infty<\kappa$ and $d\cdot p\leq 0$. Let $\mu=\min\{\lambda,1-\lambda\}$, and choose $0<\varrho<\mu\min\{\kappa,2\eta/5\}$.
    
    We claim that $z^0$ has a relatively weakly open neighbourhood $U$ in $B_Z$ such that every $z=(z_1,z_2,z_3)\in U$ satisfies
    \[
    \|z_i\|>0,\qquad z_i/\|z_i\|\in W_i,\qquad\text{and}\qquad \big|\|z_i\|-c_i\big|<\varrho.\]

Indeed, choose $h_i\in S_{c_0^*}$ with $h_i(w_i)=1$ and basic neighbourhoods 
\[
V_i=\{u\in B_{c_0} : |h(u-w_i)|<\gamma\  \text{for every $h\in F_i$}\}\subset W_i,
\]
where $\gamma>0$ and $F_i\subset B_{c_0^*}$ is finite. For sufficiently small $\delta>0$, let 
$U$ consist of points $z=(z_1,z_2,z_3)\in B_Z$ satisfying, for every $i=1,2,3$, 
\[
h_i(z_i)>c_i-\delta\qquad\text{and}\qquad |h(z_i-c_iw_i)|<\delta\qquad (h\in F_i).
\]
Fix $z=(z_1,z_2,z_3)\in U$. Denote by $\Delta_i=\|z_i\|-c_i$ ($i=1,2,3$) and $\Delta=(\Delta_1,\Delta_2,\Delta_3)$. Then $\Delta_i>-\delta$. Since $d\cdot \Delta\leq 0$, it implies that $|\Delta_i|< 4\delta$ for every $i$. Thus, when $\delta$ is small, $\|z_i\|>0$, $|\Delta_i|<\varrho$, and, for $h\in F_i$,
\[
|h(z_i/\|z_i\|-w_i)|\leq \frac{|h(z_i-c_iw_i)|+\big|c_i-\|z_i\|\big|}{\|z_i\|}<\frac{\delta+|\Delta_i|}{c_i-\delta}<\gamma.
\]

We now show that $z\in \lambda S+(1-\lambda)T$. Since $z_i/\|z_i\|\in W_i$, we have 
$u_i\in S_i$ and $v_i\in T_i$ so that
\[
\frac{z_i}{\|z_i\|}=
\frac{\lambda\widetilde a_i}{c_i}u_i
+\frac{(1-\lambda)\widetilde b_i}{c_i}v_i
\qquad(i=1,2,3).
\]
Consequently,
\[z_i=\lambda\widetilde a_iu_i+(1-\lambda)\widetilde b_iv_i+\Delta_iz_i/\|z_i\|.\]

Let $I=\{i: \Delta_i<0\}$. If $I=\varnothing$, then $\Delta_i\geq 0$ for every $i$. Since
$d_i>0$ and $d\cdot\Delta\leq 0$, it follows that $\Delta=0$.
Hence 
\[
z_i=\lambda\widetilde a_i u_i+(1-\lambda)\widetilde b_i v_i.
\]
Thus, setting $x_i=\widetilde a_i u_i$ and $y_i=\widetilde b_i v_i$ we obtain $z=\lambda x+(1-\lambda)y$, with $x\in S$ and $y\in T$. We may therefore assume that $I\neq\varnothing$. 

For $i\in I$, let 
\[
s_i=\frac{\widetilde a_i}{c_i}\Delta_i\qquad\text{and}\qquad t_i=\frac{\widetilde b_i}{c_i}\Delta_i.
\]
Let
\[
L_a=-\sum_{i\in I}d_is_i,\qquad L_b=-\sum_{i\in I}d_it_i,\qquad L=\lambda L_a+(1-\lambda)L_b.
\]
Then $L=-\sum_{i\in I}d_i\Delta_i>0$.
For $i\notin I$, define
\[
s_i=\frac{L_a}{L}\Delta_i\qquad\text{and}\qquad t_i=\frac{L_b}{L}\Delta_i.
\]
In either case, $\lambda s_i+(1-\lambda)t_i=\Delta_i$ for every $i$. Let
\[
P=\sum_{i\notin I}d_i\Delta_i.
\]
Then $P\leq L$ because $d\cdot\Delta\leq0$, and hence
\[
d\cdot s=-L_a+\frac{L_a}{L}P\leq 0\qquad\text{and}\qquad d\cdot t=-L_b+\frac{L_b}{L}P\leq 0.
\]
From $c_i=\lambda\widetilde a_i+(1-\lambda)\widetilde b_i$ and $ L=\lambda L_a+(1-\lambda)L_b$, we also have 
\[
\frac{\widetilde a_i}{c_i}\leq \frac{1}{\lambda},\qquad \frac{\widetilde b_i}{c_i}\leq \frac{1}{1-\lambda},\qquad \frac{L_a}{L}\leq\frac1\lambda, \qquad \frac{L_b}{L}\leq\frac1{1-\lambda}, 
\]
which give 
\[|s_i|\leq \frac{|\Delta_i|}{\lambda}<\frac{\varrho}{\mu}<\kappa\qquad\text{and}\qquad  |t_i|\leq \frac{|\Delta_i|}{1-\lambda}<\frac{\varrho}{\mu}<\kappa.\]
By the choice of $\kappa$, it follows that $\widetilde a+s, \widetilde b+t\in B_N^+$. 

Define $x=(x_1,x_2,x_3)$ and $y=(y_1,y_2,y_3)$ by
\[
x_i=
\begin{cases}
(\widetilde a_i+s_i)u_i,&i\in I,\\
\widetilde a_i u_i+s_i z_i/\|z_i\|,& i\notin I,
\end{cases}
\qquad\text{and}\qquad 
y_i=
\begin{cases}
(\widetilde b_i+t_i)v_i,& i\in I,\\
\widetilde b_i v_i+t_i z_i/\|z_i\|,& i\notin I.
\end{cases}
\]
Then $\lambda x+(1-\lambda)y=z$ and $x,y\in B_Z$ because $\|x_i\|\leq\widetilde a_i+s_i$ and $\|y_i\|\leq\widetilde b_i+t_i$ for every $i$. We also have, by the choice of $\varrho$,  
\[\big\|x-(\widetilde a_1u_1, \widetilde a_2u_2, \widetilde a_3u_3)\big\|\leq N(|s_1|,|s_2|,|s_3|)<\frac{5\varrho}{2\mu}<\eta\]
and 
\[
\big\|y-(\widetilde b_1v_1, \widetilde b_2v_2, \widetilde b_3v_3)\big\|\leq N(|t_1|,|t_2|,|t_3|)<\frac{5\varrho}{2\mu}<\eta.
\]
By the choice of $\eta$, this implies that $x\in S$ and $y\in T$. Thus
\[
U\subset\lambda S+(1-\lambda)T,
\]
and $Z$ has $\textup{P2}^{(2)}$.

    It remains to show that $Z$ does not have P3.
    Fix $f_1,f_2,f_3\in S_{c_0^*}$ and $\alpha\in (0,1/10)$. Consider the slices $S_1=S(B_Z,(f_1,0,0),\alpha)$, $S_2=S(B_Z,(0,f_2,0),\alpha)$, and $S_3=S(B_Z,(0,0,f_3),\alpha)$. If $x=(x_1,x_2,x_3)\in S_i$, then $\|x_j\|<\alpha$ whenever $j\neq i$. Hence every $z=(z_1,z_2,z_3)\in (S_1+S_2+S_3)/3$ satisfies $\|z_j\|\leq \frac{1+2\alpha}{3}$ $(j=1,2,3$), and therefore $\|z\|_N\leq \frac{5(1+2\alpha)}{6}<1$. Thus $(S_1+S_2+S_3)/3$ does not intersect $S_Z$. So $Z$ fails P3. 
\end{proof}

\section{Strict separation of P1 and \texorpdfstring{$\overline{\text{P1}}$}{P1}}
For $n\in\mathbb N$, let $X_n$ be the two-dimensional Banach space $\mathbb R^2$ with the norm
\[
{\|(s,t)\|}_n=\max\Big\{|t|,\frac{n|s|+|t|}{n+1}\Big\}.\tag{1\label{eq:1}}
\]
Then $B_{X_n}=\operatorname{conv}([-1,1]^2\cup\{\pm(1+\frac1n,0)\})$.

Define \[X=\{x=(x_n) : x_n\in\mathbb R^2,\ x_\infty\coloneqq \lim_n x_n\ \text{exists in $\mathbb R^2$}\}\] and $\|x\|=\sup_n {\|x_n\|}_n$. For every $u\in\mathbb R^2$, 
\[\frac{n}{n+1}\|u\|_\infty \leq \|u\|_n\leq \|u\|_\infty\qquad\text{and}\qquad \|u\|_n\to\|u\|_\infty.\]
Therefore $X$ is isomorphic to $c(\mathbb R^2)$, and hence to $c_0$. 
We shall use the following description of the dual $X^*$. Every $f\in X^*$ has a unique representation \[f(x)=f_\infty(x_\infty)+\sum_{n=1}^\infty f_n(x_n),\] where $f_\infty\in (\ell_\infty^2)^*=\ell_1^2$ and $f_n\in X_n^*$, $\sum_{n=1}^\infty\|f_n\|<\infty$, and $\|f\|=\|f_\infty\|_1+\sum_{n=1}^\infty\|f_n\|$.
Thus $X^*=(\ell_1^2\oplus\bigoplus_{n=1}^\infty X_n^*)_{\ell_1}$.

The following result shows that $\overline{\textup{P1}}$ is strictly weaker than $P1$ even within the class of Banach spaces covered by \cite[Theorem~2.7]{MR4116184}. This also gives a negative answer, for the precise $\ell_1$-type dual structure considered there, to the question posed in \cite[Introduction]{MR4116184}.

\begin{proposition}\label{prop:4.1}
    The space $X$ has $\overline{\textup{P1}}$ but not \textup{P1}.
\end{proposition}
    
\begin{proof}
Every summand of $X^*$ is a two-dimensional real Banach space and therefore has property $\textup{(co)}$ by \cite[Proposition~3.7]{MR4034749}. It follows from \cite[Theorem~2.7]{MR4116184} that $X$ has $\overline{\textup{P1}}$.

We show that $X$ fails P1. Define $f\in X^*$ by $f(x)=(x_\infty)_2$, the second coordinate of $x_\infty$. Then $\|f\|=1$. Consider the opposite slices $S=S(B_X,f,1)=\{x\in B_X : f(x)>0\}$ and $-S=S(B_X,-f,1)=\{x\in B_X : f(x)<0\}$, and let $C=\frac12S+\frac12(-S)$. Let $x,y,z\in B_X$ be the constant sequences with respective values $(1,0)$, $(1,1)$, and $(1,-1)$. Then $y\in S$ and $z\in -S$, and $x=\frac12 y+\frac12z\in C$. 

For $m\in\mathbb N$, define $x^m=(x^m_n)\in X$ by \[x^m_n=\begin{cases}(1,0),& n< m,\\(1+1/n,0),& n\geq m.\end{cases}\] 
Then $x^m\in B_X$ and 
    \[
    \|x^m-x\|=\sup_{n\geq m}\|(1/n,0)\|_n=\frac1{m+1}\to 0.
    \]
    Suppose that $x^m=(y+z)/2$ for some $y=(y_n), z=(z_n)\in B_X$. The point $(1+1/n,0)$ is an extreme point of $B_{X_n}$. Hence $y_n=z_n=(1+1/n,0)$ for all $n\geq m$. And passing to the limit gives $y_\infty=z_\infty=(1,0)$, so $f(y)=f(z)=0$. Thus $x^m\notin C$. 
    
    If $C$ were relatively weakly open in $B_X$, then, since $x\in C$ and $x^m\to x$ in norm (and hence weakly), we would have $x^m\in C$ for all sufficiently large $m$, which is a contradiction. Thus, $X$ does not have P1.
\end{proof}

\begin{remark}
    Proposition~\ref{prop:4.1} applies inside the precise $\ell_1$-type dual framework of \cite[Theorem~2.7]{MR4116184}, but still separates P1 and $\overline{\textup{P1}}$ (see the discussion on \cite[p.~2]{MR4116184}). It also separates, for general Banach spaces, conditions (ii) and (v) in \cite[Corollary~2.4]{MR4422399}, where (ii) is P1 and (v) is $\overline{\textup{P1}}$. By Theorem~\ref{thm:thm2.1}, the example in Proposition~\ref{prop:4.1} also separates CWO and $\overline{\textup{CWO}}$.
\end{remark}

\bibliographystyle{amsplain}
\bibliography{references}
\end{document}